\documentclass{amsart}

\usepackage{amsmath,amsthm,amsfonts,amssymb,latexsym,appendix,hyperref}

\newtheorem{thm}{Theorem}[section]
\newtheorem{thm*}{Theorem}
\newtheorem{lemma}[thm]{Lemma}
\newtheorem{cor}[thm]{Corollary}
\newtheorem{lemma*}[thm*]{Lemma}

\newtheorem{problem}[thm]{Problem}

\theoremstyle{definition}
\newtheorem{remark}[thm]{Remark}
\newtheorem{defn}[thm]{Definition}

\DeclareMathOperator{\inter}{int}
\DeclareMathOperator{\cl}{cl}
\DeclareMathOperator{\diam}{diam}

\newcommand{\from}{\colon}

\newcommand{\bfSigma}{\mathbf{\Sigma}}
\newcommand{\boolH}{\mathcal{B}} 

\newcommand{\N}{\mathbb{N}}
\newcommand{\R}{\mathbb{R}}
\newcommand{\Z}{\mathbb{Z}}
\newcommand{\T}{\mathbb{T}}

\renewcommand{\subset}{\subseteq}

\newcommand{\boundary}{\partial}
\newcommand{\boundaryinfty}{\partial_{\text{vis}(\infty)}}

\newcommand{\union}{\cup}
\newcommand{\bigunion}{\bigcup}
\newcommand{\inters}{\cap}

\newcommand{\define}[1]{\emph{#1}} 
\newcommand{\actson}{\curvearrowright}

\newcommand{\norm}[1]{\lvert {#1} \rvert}

\newcommand{\bfu}{\mathbf{u}}

\begin{document}

\title{Borel Circle Squaring}

\thanks{The first author was partially supported by NSF grant 
DMS-1500974.}

\author{Andrew S. Marks}
\address{Department of Mathematics, University of California at Los Angeles}
\email{marks@math.ucla.edu}

\author{Spencer T. Unger}
\address{Department of Mathematics, University of California at Los Angeles}
\email{sunger@math.ucla.edu}

\date{\today}

\begin{abstract}
  We give a completely constructive solution to Tarski's circle squaring
  problem. More generally, we prove a Borel version of an 
  equidecomposition theorem due to Laczkovich. If $k \geq 1$ and $A, B
  \subset \R^k$ are bounded Borel sets with the same positive Lebesgue measure whose
  boundaries have upper Minkowski dimension less than
  $k$, then $A$ and $B$ are equidecomposable by translations using Borel
  pieces. This answers a
  question of Wagon. Our proof uses  ideas from the study of flows in
  graphs, and a recent result of Gao, Jackson, Krohne, and Seward on
  special types of witnesses to the hyperfiniteness of free Borel actions
  of $\Z^d$. 
\end{abstract}

\maketitle

\section{Introduction}

In 1925, Tarski posed the problem of whether a disk and square of the same
area in the plane are equidecomposable by isometries \cite{T}. That is, can
a disk be partitioned into finitely many pieces which can be rearranged by
isometries to partition a square of the same area? This problem became
known as Tarski's circle squaring problem. In contrast to the Banach-Tarski
paradox in $\R^3$, a theorem of Tarski (see \cite{W}) implies that any two
Lebesgue measurable sets in $\R^2$ that are equidecomposable by isometries
must have the same Lebesgue measure, even when the pieces used in the
equidecomposition are allowed to be nonmeasurable. 
Thus, the requirement that the circle and the square have the same area is
necessary.

The idea of comparing the measure of sets by partitioning them into
congruent pieces has a long history, dating back in some form to Euclid.
The well known Wallace-Bolyai-Gerwien theorem states that two polygons in
$\R^2$ have the same area if and only if they are \define{dissection
congruent}, that is, equidecomposable by polygonal pieces where we may
ignore boundaries. Hilbert's third problem asked whether any two polyhedra
of the same volume are dissection congruent. Dehn famously gave a negative
answer to this problem. 

Early work on Tarski's circle squaring problem established the
nonexistence of certain types of equidecompositions. 
Dubins, Hirsch, and Karush
\cite{DHK} introduced the notion of \define{scissors congruence} in $\R^2$,
considering equidecompositions using pieces whose boundaries consist of a
single Jordan curve. They showed that a disk in $\R^2$ is scissors
congruent to no convex set other than translates of itself. So Tarski's
circle squaring problem has a negative answer for this restrictive
type of equidecomposition. Gardner~\cite{G} also showed that Tarski's circle
squaring problem cannot be solved using any locally discrete subgroup of
isometries. 

Laczkovich answered Tarski's question positively in 1990 \cite{L90}, using
only translations 
in his equidecomposition. 
In 1992, he 
improved this result to give a very general sufficient condition for when two
bounded sets in $\R^k$ of the same Lebesgue measure are equidecomposable by
translations. If $X \subset \R^k$, then we let $\boundary X =  \cl(X)
\setminus \inter(X)$ indicate its boundary and $\Delta(X)$ be the
\define{upper Minkowski dimension} of $X$ (see \cite{L92}). Let
$\lambda$ be Lebesgue measure. 

\begin{thm}[{\cite[Theorem 3]{L92}}]\label{L_thm}
  Suppose $k \geq 1$ and suppose $A, B \subset \R^k$ are bounded sets such
  that $\lambda(A) = \lambda(B) > 0$,  $\Delta(\boundary A) < k$, and
  $\Delta(\boundary B) <
  k$. Then $A$ and $B$ are equidecomposable by translations.
\end{thm}

Laczkovich's proofs in \cite{L90} and \cite{L92} are nonconstructive and use the axiom of choice.
It remained an open problem whether such equidecompositions could be
done constructively. 
In \cite[Appendix C, Question 2.a]{W}, Wagon made this question
precise by asking whether 
Tarski's circle squaring problem could be solved using
Borel pieces. Recall that the Borel sets are the smallest collection of
sets obtained by starting with the open sets and 
closing under the operations of countable union,
countable intersection, and complementation. 

In a recent breakthrough,
Grabowski, M\'ath\'e, and Pikhurko \cite{GMP} showed that Tarski's circle
squaring problem can be solved using 
Lebesgue measurable or Baire measurable pieces, by proving a version
of Theorem~\ref{L_thm} for Lebesgue measurable/Baire measurable equidecompositions.
Their proof is also non-constructive since it uses the axiom of choice to
construct the equidecomposition on a null/meager set. However, their result
gave strong evidence that a constructive solution to Tarski's circle
squaring problem might exist, since it showed that there cannot be any measure-theoretic
or Baire category obstruction. 

In this paper, we answer Wagon's question~\cite[Appendix C, Question
2.a]{W} and give a completely constructive solution to Tarski's circle
squaring problem. More generally, we prove a Borel version of Laczkovich's
Theorem~\ref{L_thm}. This generalizes the
results of \cite{GMP}. It also provides a ``Borel solution'' to Hilbert's third
problem: any two bounded Borel sets in $\R^k$ with ``small boundary'' have
the same measure if and only if they are translation
equidecomposable using Borel pieces: 

\begin{thm}\label{main_thm}
  Suppose $k \geq 1$ and suppose $A, B \subset \R^k$ are bounded Borel sets such
  that $\lambda(A) = \lambda(B) > 0$, $\Delta(\boundary A) < k$, and
  $\Delta(\boundary B) <
  k$. Then $A$ and $B$ are equidecomposable by translations using Borel pieces.
\end{thm}

The pieces that we use in our equidecomposition are quite simple. Let
$\bfSigma^{A,B}_{1}$ be the collection of all open balls in $\R^k$,
translates of $A$, and
translates of $B$. Then inductively, let $\boolH^{A,B}_n$ be all finite
Boolean combinations of $\bfSigma^{A,B}_n$ sets, and let $\bfSigma^{A,B}_{n+1}$ be
all countable unions of sets in $\boolH^{A,B}_n$. If $A$ and $B$ are
$\bfSigma^0_m$ in the usual Borel hierarchy, then clearly
every set in $\bfSigma^{A,B}_n$ is $\bfSigma^{0}_{n+m-1}$.
The pieces we use in our
equidecomposition are sets in $\boolH^{A,B}_4$ (see
Section~\ref{complexity_section}). If $A$ and $B$ are a disk and
a square of the same area in $\R^2$, then it is easy to see that $A$ and
$B$ are not equidecomposable using set in $\boolH^{A,B}_1$, since $A$ and $B$
are not scissors congruent. 

A key idea in our proof is to use flows in infinite graphs as an
intermediate step towards constructing equidecompositions.
Under the hypotheses of
Theorem~\ref{main_thm}, in Section~\ref{flow_section}
we give an explicit and simple construction of a
bounded ``Borel flow'' between $A$ and $B$.  Laczkovich's discrepancy
estimates from \cite{L92} -- the central ingredient in the proof of
Theorem~\ref{L_thm} -- are used to show the convergence of this
construction.  

Another important tool in our proof comes from the theory of orbit
equivalence and Borel equivalence relations. In particular, we use a
result of Gao, Jackson, Krohne, and Seward, about special types of witnesses to the hyperfiniteness of free
Borel actions of $\Z^d$ (see Theorem~\ref{cake}). 
Their theorem is part of an ongoing research program to understand the
complexity of actions of amenable groups in descriptive set theory and
ergodic theory. It builds on the result due to Weiss that every free Borel
action of $\Z^d$ is hyperfinite, and more recent work of Gao-Jackson
\cite{GJ} (see also \cite{SS}). 
Gao, Jackson, Krohne, and Seward's theorem is announced in \cite{GJKS}, but has not yet appeared, and so we 
include a proof of their result in 
Appendix~\ref{cake_appendix} for completeness. (This proof is different
from their forthcoming proof). So our paper is essentially self-contained
except for our use of Laczkovich's discrepancy estimates.

In Section~\ref{integral_section} we use this hyperfiniteness witness to turn
our real-valued Borel flow between $A$ and $B$ into an integer-valued flow.
This step in our proof also relies on the integral flow theorem which is a
corollary of the Ford-Fulkerson proof of the max-flow min-cut theorem.
The last step in our proof in Section~\ref{main_proof_section} uses this  
integer valued flow to define a Borel
equidecomposition from $A$ to $B$. 

These ideas are very different from the
work of \cite{GMP}. The tools they use are quite specific to the
measurable and Baire measurable settings and cannot easily be adapted to
prove Theorem~\ref{main_thm}.

The authors would like to thank Anton Bernshteyn, Clinton Conley, Steve
Jackson, Alekos Kechris, Igor Pak, Robin Tucker-Drob, and Brandon Seward for
helpful discussions.

\section{Preliminaries}
\label{sec:defn}

If $a \from \Gamma \actson X$ is an action of a group $\Gamma$ on a set
$X$, then $A, B \subset X$ are said to be \define{$a$-equidecomposable} if
there exist a partition $\{A_1, \ldots, A_n\}$ of $A$ and group elements
$\gamma_1, \ldots, \gamma_n \in \Gamma$ such that $\gamma_1 \cdot A_1,
\ldots, \gamma_n \cdot A_n$ is a partition of $B$. Similarly we say that
$A, B \subset X$ are \define{$a$-equidecomposable using Borel pieces} if
there exist a partition $\{A_1, \ldots, A_n\}$ of $A$ into Borel sets and group
elements $\gamma_1, \ldots, \gamma_n \in \Gamma$ such that $\gamma_1 \cdot
A_1, \dots, \gamma_n A_n$ is a partition of $B$ into Borel pieces. 

Suppose $A, B \subset \R^k$ are bounded, and we wish to show that $A$ and
$B$ are equidecomposable by translations using Borel pieces. By scaling and
translating $A$ and $B$, we may assume that $A, B \subset [0,1/2)^k$ which
is a subset of the $k$-torus $\T^k = \R^k/\Z^k$ which we identify with
$[0,1)^k$. Then it is clear that any equidecomposition by translations
between $A$ and $B$ in $\T^k$ can also be done in $\R^k$ using the same set
of pieces. This idea was used by Laczkovich~\cite{L90}. We will work in
$\T^k$ throughout the paper and show that $A$ and $B$ are equidecomposable
by translations using Borel pieces in $\T^k$. 

$\T^k$ inherits both its topology and abelian group structure from $\R^k$. We
let $\lambda$ be Haar measure on $\R^k/\Z^k$ which we can identify with Lebesgue
measure on the fundamental domain $[0,1)^k$.  If $F \subset \T^k$ is finite and
$A \subset \T^k$ is $\lambda$-measurable, then the \define{discrepancy of $F$
relative to $A$} is \[D(F,A) = \left| |F \inters A|/|F| - \lambda(A) \right|.\]

Given $\bfu = (u_1, \ldots, u_d) \in (\T^k)^d$, let $a_{\bfu}$ be the
action of $\Z^d$ on $\T^k$ defined by 
\[(n_1, \ldots, n_d) \cdot_{a_\bfu} x = n_1 u_1 +
\ldots + n_d u_d + x\] 
for $(n_1, \ldots, n_d) \in \Z^d$ and $x \in \T^k$. Let 
\[R_{N} = \{(n_1, \ldots, n_d) \in \Z^d : 0 \leq n_i
< N \text{ for every $i \leq d$}\}\]
and let the image of this set under the action $a_\bfu$ be
\[F_{N}(x, a_\bfu) = R_{N} \cdot_{a_{\bfu}} x.\]

Laczkovich proved the following crucial estimate on the discrepancies of
these sets, using ideas from Diophantine approximation, and building on
work of Schmidt \cite{S} and Niederreiter and Wills \cite{NW}.
\begin{lemma}[Laczkovich {\cite[Proof of Theorem 3]{L92}}, see also
{\cite[Lemma 6]{GMP}}]\label{L_discrep}
  Suppose $A \subset \T^k$ is measurable, $\Delta(\boundary A) < k$ and
  $\lambda(A) > 0$. Let $d$ be such that $d > 2k /(k - \Delta(\boundary
  A))$. Then for almost every 
  $\bfu \in (\T^k)^d$, there is an $\epsilon > 0$ and $M > 0$ such
  that for every $x \in \T^k$ and $N > 0$,
  \[D(F_N(x,a_\bfu),A) \leq M N^{-1 - \epsilon}.\]
\end{lemma}

The variables $u$, $(x_1, \ldots, x_d)$, and $\eta$ in \cite{L92}
correspond to our $x$, $\bfu$, and $\epsilon$, respectively. 

Though we will not need this observation for our proof, we remark that the
order of the quantifiers over $\bfu$ and $A$ can be reversed here.  Almost every
$\bfu \in (\T^k)^d$ satisfies the lemma for every measurable $A \subset \T^k$.
This follows from Laczkovich's argument.

By a graph $G$ on a set $V$, we mean a (simple undirected) graph with
vertex set $V$, so the edge relation of $G$ will be a symmetric irreflexive
relation on $V$. If $G$ is a graph on a vertex set $V$ and $x \in V$, then
we write $[x]_G$ for the set of vertices in the same connected component as
$x$. We let $d_G$ be the graph metric on the vertex set of $G$. We will write $d$ instead of $d_G$ when the
graph $G$ is clear from context.
Let $N_G(x)$ be the set of neighbors of $x$, so $N_G(x) = \{y \in V \colon
d_G(x,y) = 1\}$. We will write $N(x)$ when the
graph is clear.  A graph $G$ is said to be \define{locally
finite} if $N_G(x)$ is finite for every vertex $x$ of $G$. If $U \subset V$,
then we let $G \restriction U$ be the induced subgraph of $G$ on the set
$U$. If $F$ is a set of edges of $G$, we let $G - F$ be the subgraph of $G$
obtain by removed the edges in $F$. 

If $a \from \Z^d \actson X$ is an action of $\Z^d$ on a set $X$, let $G_a$
be the graph with vertex set $X$ where there is an edge from $x$ to $y$
if $\gamma \cdot x = y$ for some $\gamma \in \Z^d$ with $\norm{\gamma}_\infty =
1$. Here $|(\gamma_1, \ldots, \gamma_d) |_\infty = \sup_i |\gamma_i|$ is
the sup norm. The bulk of our proof is establishing the existence of
certain types of flows (defined in Section~\ref{flow_defn_section}) on the graph $G_{a_{\bfu}}$ where $a_{\bfu}$ is as
in Lemma~\ref{L_discrep}.

Recall that a \define{standard Borel space} is a set $X$ equipped with a
$\sigma$-algebra generated by a Polish (separable, completely metrizable)
topology on $X$. The space $\T^k$ equipped
with its Borel sets is an example of a standard Borel space. If $X$ is a
standard Borel space and $n > 0$ then we equip $X^n$ with the standard Borel space
arising from the product topology on $X^n$. We will
use $[X]^{< \infty}$ to note the standard Borel space of all finite subsets
of $X$. Here we use the standard Borel structure where $B \subset [X]^{< \infty}$ is
Borel if for every $n$, $\{(x_1, \ldots, x_n) \in X^n : \{x_1, \ldots,
x_n\} \in B\}$ is a Borel subset of $X^n$. 

In some of our proofs, we will need to make an arbitrary choice between
finitely many elements of some standard Borel space $X$. In this situation
we will fix some Borel linear ordering $<$ on $X$, and choose the $<$-least
element. Note that every standard Borel space admits a Borel linear
ordering (one can see this using the isomorphism theorem for standard Borel
spaces and the fact that the usual linear ordering on $\R$ is Borel
\cite[Theorem 15.6]{K}). In the case when $X = \T^k$, we may simply use the
lexicographic ordering on $[0,1)^k$. 

A \define{Borel graph} is a graph whose vertices are the elements of a
standard Borel space $X$, and whose edge relation is Borel as a subset of
$X \times X$. For a recent survey of the theory of Borel
graphs, see \cite{KM}. If $\bfu \in (\T^k)^d$, then the graph
$G_{a_{\bfu}}$ is an example of a Borel graph; since the action $a_{\bfu}$
is continuous, the edge relation of $G_{a_{\bfu}}$ is closed. In order to prove
Theorem~\ref{main_thm}, we will not need to consider any Borel graphs other
than $G_{a_{\bfu}}$. However, some of our lemmas are stated in the
generality of 
any Borel graph of the form $G_a$, where $a$ is a free Borel action of
$\Z^d$.
If $G$ is a Borel graph
on $X$, we let $[G]^{< \infty}$ be the set of all finite subsets of $X$
that lie in a single connected component of $G$. This is a Borel subset of
$[X]^{< \infty}$.

\section{Flows in graphs}\label{flow_defn_section}
Our proof will use the following types of flows on graphs.
Suppose $G$ is a locally finite graph with vertex set $V$, and $f:V \to \R$. An
\define{$f$-flow} on $G$ is a real-valued function $\phi$ on the edges
of $G$ such that $\phi(x,y) = - \phi(y,x)$ for every edge $(x,y)$ of $G$, and
such that for every $x \in V$,
\[f(x) = \sum_{y \in N(x)} \phi(x,y).\]

If $\epsilon > 0$, then an \define{$(\epsilon,f)$-flow} is a real-valued
function $\phi$ on the edges of $G$ such that
$\phi(x,y) = - \phi(y,x)$ for every edge $(x,y)$ of $G$ and such that for every $x \in V$,
\[\left| f(x) - \sum_{y \in N(x)} \phi(x,y) \right| < \epsilon.\]

Suppose that $c$ is a nonnegative function on the edges of $G$ (where we
may have $c(x,y) \neq c(y,x)$). We call $c$ a
\define{capacity function} on $G$ and we say that an $f$-flow $\phi$
\define{is bounded by $c$} if $\phi(x,y) \leq c(x,y)$ for every edge
$(x,y)$ in $G$. We say that an $f$-flow $\phi$ is \define{bounded} if it is
bounded by a constant capacity function.

If $\phi_1, \ldots,
\phi_n$ are $f$-flows on a graph $G$, then their average $\frac{1}{n}(\sum_i
\phi_i)$ is also an $f$-flow. The average of finitely many $(\epsilon,f)$-flows is similarly an $(\epsilon,f)$-flow. At
a key step in our proof of Lemma~\ref{flow_lemma}, we will need to average
over flows in this way. This is the reason why we have to use 
real-valued flows (instead of matchings) in our proof. 

A folklore restatement of the max-flow min-cut theorem characterizes 
exactly when a finite graph $G$ admits an $f$-flow
bounded by a capacity function $c$. Roughly, for every finite set $F$ of
vertices, $\sum_{x \in F} f(x)$ should be
at most the capacity of the edges leaving $F$. 

\begin{thm}\label{finite_flow_existence}
  Suppose $G$ is finite graph on $X$, $f \from X \to \R$ is a
  function, and $c$ is a capacity function for $G$. 
  Then $G$ has an $f$-flow
  bounded by $c$ if and only if for every set $F \subset X$, 
  \[ -\sum_{\{(x,y) \in G : x \notin F \land y \in F\}} c(x,y)
  \leq \sum_{x \in F} f(x) \leq \sum_{\{(x,y) \in G : x \in F \land y \notin
  F\}} c(x,y).\]
\end{thm}

\begin{proof}
The forward direction is clear, so we focus on the reverse. 
Define $G'$ to be the finite
graph containing $G$ as a subgraph where we add two vertices to $X$, a source $s$ and a sink $t$, and edges
as follows. Add an edge $(s,x)$ to each $x \in X$ such that $f(x) > 0$ and add
an edge $(y,t)$ to each $y \in X$ such that $f(y) < 0$. Define a capacity
function $c'$ for $G'$ as follows. 
Let $c'(x,y) = c(x,y)$ for every edge $(x,y)$
in $G$. For every edge $(s,x)$ incident to $s$, let $c'(s,x) = f(x)$ and
for every edge $(y,t)$ incident to $t$, let $c'(y,t) = - f(y)$ and $c'(t,y) =
0$. 
Now apply the max-flow min-cut theorem \cite{D} to the graph
$G'$ with source $s$, sink $t$, and capacity function $c'$.  

Take any $F \subset X$, and let $\{\{s\} \union F,\{t\} \union (X \setminus
F)\}$ be a cut in $G'$. One can show  using either of the two inequalities
on $f$ and $c$ that the capacity of this cut is greater than or equal to 
\[ \sum_{\{x : f(x) > 0 \}} f(x) = - \sum_{\{x : f(x) < 0 \}} f(x). \]
which is the capacity of the cut containing just $s$ or just $t$. This
value is therefore the minimum capacity of a cut in $G'$, and
the restriction to $G$ of the resulting maximal flow on $G'$ is as required.
\end{proof}

For finite graphs, there is no real difference in studying $f$-flows as 
defined above, and flows with a single source and single sink. This is
because one can easily convert between these types of problems using the
idea in the proof of Theorem~\ref{finite_flow_existence} above. However,
for infinite graphs, it becomes complicated to try and use a single source
and sink to model an $f$-flow. (See for instance \cite{ABGPS}). For
example, if $G$ is the infinite $3$-regular tree and $f$ is the constant function $1$
on the vertices of $G$, then $G$ admits an $f$-flow bounded by $1$. (Even
though every vertex in the graph is a source, and there are no sinks).

In our proof, we will be constructing flows on infinite 
graphs. We remark that by taking ultralimits, one can characterize when a
locally finite
graph admits an $f$-flow bounded by a given capacity function. Note that in
this case, we need \emph{both} inequalities given in
Theorem~\ref{finite_flow_existence}.
We
mention this result as a contrast to some of our results about Borel flows
in infinite graphs. However, we will not use
Theorem~\ref{infinite_flow_existence} in our proof of
Theorem~\ref{main_thm}. 

\begin{thm}[Folklore] \label{infinite_flow_existence}
  Suppose $G$ is locally finite graph on $X$, $f \from X \to \R$ is a
  function, and $c$ is a capacity function such that $c(x,y) < \infty$ for
  every edge $(x,y)$ in $G$. Then $G$ has an $f$-flow
  bounded by $c$ if and only if for every finite set $F \subset X$, 
  \[ -\sum_{\{(x,y) \in G : x \notin F \land y \in F\}} c(x,y)
  \leq \sum_{x \in F} f(x) \leq \sum_{\{(x,y) \in G : x \in F \land y \notin
  F\}} c(x,y)\]
\end{thm}
\begin{proof}
By working in each connected
component separately, we can assume that $G$ is connected. (This uses
the axiom of choice). 
Fix a vertex $x$
in $G$ and for each $n \geq 0$ consider the graph $G_n$ which is the
induced subgraph of $G$ on the vertex set $\{ y : d_G(x,y) \leq n \}$. Let
$G_n'$ be the graph obtained by adding a single new vertex to $G_n$ and
connecting it to each of the vertices $\{y : d_G(x,y) = n\}$ in $G_n$. Let
$f_n$ be the function on the vertices of $G_n'$ which is equal to $f$ on
$\{y : d_G(x,y) \leq n\}$ and
equal to $- \sum_{\{y : d_G(x,y) \leq n\}} f(y)$ at the new vertex.
Clearly $G_n'$ and $f_n$ satisfy the hypotheses of
Theorem~\ref{finite_flow_existence} for the capacity function $c'$ which is
equal to $c(x,y)$ on edges in $G_n$, and is infinite on the new edges in
$G_n'$. This is because the only sets
$F$ yielding finite total capacities either lie in the interior of the
$n$-ball around $x$, or their complement does.

Let $\phi_n$ be an $f_n$-flow for $G_n'$. Let $U$ be a nonprincipal
ultrafilter on $\N$. Define $\phi$ on the edges of $G$ by the ultralimit $\phi(e) = \lim_U
\phi_n(e)$. The limit converges since $|\phi_n(e)| \leq c(e)$.
It is straightforward to check that $\phi$ is an $f$-flow for $G$. 
\end{proof}

\section{Constructing flows in $G_a$}
\label{flow_section}

For this section, we fix a free action $a \from \Z^d \actson X$ and a
function $f \from X \to \R$. Assuming $f$ satisfies the conditions given in
Lemma~\ref{flow_lemma}, we give an explicit construction of an $f$-flow in
the graph $G_a$.  To prove Theorem~\ref{main_thm}, we will eventually apply
Lemma~\ref{flow_lemma} when the action $a$ is $a_{\bfu}$ for some $\bfu$
satisfying Lemma~\ref{L_discrep}, and $f = \chi_A - \chi_B$, the difference
between the characteristic functions of $A$ and $B$.  Lemma~\ref{L_discrep}
ensures that $f$ satisfies the hypothesis of Lemma~\ref{flow_lemma}. In
this situation, the flow given by Lemma~\ref{flow_lemma} will clearly be
Borel (see Section~\ref{complexity_section}).

For every $i > 0$, let $\pi_i \from \Z^d/(2^i \Z)^d \to \Z^d/(2^{i-1}
\Z)^d$ be the canonical homomorphism. This yields the inverse limit
\[\hat{\Z^d} = \varprojlim_{i \geq 0} \Z^d/(2^i\Z)^d\] where elements of
$\hat{\Z^d}$ are sequences $(h_0, h_1, \ldots)$ such that $\pi_{i}(h_{i}) =
h_{i-1}$ for all $i > 0$.

We begin by making several definitions which depend on $a$ and $f$. 
Below, if $F \subset X$ is finite,
we will write $\sum_F f$ instead of $\sum_{x \in F} f(x)$.

Suppose $x \in X$ and $h \in \Z^d/(2^n \Z)^d$ for some $n > 0$. View
$h$ as a coset of $2^n \Z^d$ and recall that 
$R_{N} = \{(n_1, \ldots, n_d) \in \Z^d : 0 \leq n_i
< N \text{ for every $i \leq d$}\}$. If we let
\[P_{x,h} = \left\{(h' + R_{2^n}) \cdot x : h' \in h\right\},\]
then $P_{x,h}$ is a partition of $[x]_{G_a}$ into sets of size $2^{nd} =
|R_{2^n}|$. For every $y \in [x]_{G_a}$, let $[y]_{(x,h)}$ be the element
of $P_{x,h}$ that contains $y$. Let 
\[Q_{x,h}(y) = \{ [z]_{(x,\pi_n(h))} \colon z \in [y]_{(x,h)} \},\] 
so $Q_{x,h}(y)$ is a partition of $[y]_{(x,h)}$ into $2^d$ many pieces.

Suppose $\gamma \in \Z^d$ is such that $\norm{\gamma}_\infty = 1$. Let
$U_{x,h}(y, \gamma)$ be the set of $z \in [y]_{(x,h)}$
such that $(2^{n-1} \gamma) \cdot
z \in [y]_{(x,h)}$ and there exists some
$0 \leq i < 2^{n-1}$ such that $y =
(i \gamma) \cdot z$. Note that all $z \in U_{x,h}(y, \gamma)$ must
come from a unique element of $Q_{x,h}(y)$, which we will call $Q_{x,h}(y,\gamma)$.
Let $n_{x,h}(y,\gamma \cdot y) = |U_{x,y}(y, \gamma)|$. 

Define a function $\phi_{x,h}$ on the edges of $G_a \restriction [x]_{G_a}$ by
\[\phi_{x,h}(y,\gamma \cdot y) = \frac{1}{2^{nd}} n_{x,h}(y,\gamma\cdot y)
\sum_{Q_{x,h}(y,\gamma)} f.\]
The rough idea is that $\phi_{x,h}$ is defined by working inside each element of
$P_{x,h}$ and for each $z$ in this set, $\phi_{x,h}$ 
moves some mass at $z$ along a path of the
form \linebreak $z, \gamma \cdot z, \ldots, (2^{n-1} \gamma) \cdot z$ to the point
$(2^{n-1} \gamma) \cdot z$. The definition of
$\phi_{x,h}$ at an edge \linebreak $(y, \gamma \cdot y)$ comes from summing the
contribution of all the different $z$'s whose associated path includes $(y,
\gamma \cdot y)$.

Finally, for all $(h_0, h_1, \ldots) \in \hat{\Z^d}$, $y \in
  [x]_{G_a}$, and $\gamma \in \Z^d$ with $\norm{\gamma}_\infty = 1$, let 
\[\psi_{x,(h_0, \ldots, h_n)}(y,\gamma \cdot y) = 
\sum_{i = 1}^n \left(\phi_{x,h_i}(y,\gamma \cdot y) -
  \phi_{x,h_i}(\gamma \cdot y,y) \right).\]

Our definition of $\psi_{x,(h_0, \ldots, h_n)}$ is chosen so that 
the 
``error'' at $y$ which prevents $\psi_{x,(h_0, \ldots, h_n)}$ from being a
flow is exactly the average value of $f$ over $[y]_{(x,h_n)}$. This is the
content of Lemma~\ref{error_lemma}.

\begin{lemma}\label{error_lemma}
For every $x \in X$, $(h_0, h_1 \ldots) \in \hat{\Z^d}$, $n > 0$ and 
$y \in [x]_{G_a}$,
  \[f(y) - \sum_{\norm{\gamma}_\infty = 1} \psi_{x,(h_0, \ldots, h_n)}(y,\gamma
  \cdot y) =
  \frac{1}{2^{nd}} \sum_{[y]_{(x,h_n)}} f\]
  \end{lemma}

\begin{proof}
We work by induction. For the base case of $n = 0$, the left hand size is $f(y)$
since the summation defining $\psi_{x,(h_0)}$ is empty. The right hand
side is also $f(y)$ since $[y]_{(x,h_0)} = \{y\}$. 

Let the left hand size of the above equation be 
\[\theta_n(y) = f(y) - \sum_{\norm{\gamma}_\infty = 1}\psi_{x,(h_0, \ldots,
h_n)}(y,\gamma \cdot y).\]
Note that
\[
\theta_n(y) = \theta_{n-1}(y) - \left( \sum_{\norm{\gamma}_\infty = 1} \phi_{x,h_n}(y,
\gamma \cdot y) -\phi_{x,h_n}(-\gamma \cdot y,y) \right)
\]
where we have changed $\phi_{x,h_n}(\gamma \cdot y,y)$ to $\phi_{x,h_n}(-
\gamma \cdot y, y)$ in the second term in the summation, by using the fact
that we are summing over all $\gamma$ with $\norm{\gamma}_\infty = 1$. Note
that $(-\gamma \cdot y, y) = (-\gamma \cdot y, \gamma \cdot (- \gamma)
\cdot y)$ is an edge oriented in the ``direction'' of $\gamma$.

We now compute this sum. Fix $z \in U_{x,h_n}(y,\gamma)$ so that $z \in
[y]_{(x,h_n)}$, $2^{n-1}\gamma \cdot z \in [y]_{(x,h_n)}$ and $y =
(i\gamma)\cdot z$ for some $0 \leq i <2^{n-1}$.
If $y \neq z$, then $z$ is counted in both
$n_{x,h_n}(y, \gamma\cdot y)$ and $n_{x,h_n}(-\gamma\cdot y, y)$.  Moreover it
contributes the same amount to $\phi_{x,h_n}(y,\gamma\cdot y)$ and
$\phi_{x,h_n}(-\gamma\cdot y, y)$.  Using this fact to cancel corresponding
terms and summing over $\gamma$, we get
\begin{multline*}
\sum_{\norm{\gamma}_\infty =1} \phi_{h_n,x}(y, \gamma \cdot y)
-\phi_{h_n,x}(-\gamma \cdot y,y)  \\ 
= \frac{1}{2^{nd}} \left( |S_y| \sum_{[y]_{(x,h_{n-1})}} f -
\sum_{z \in \{(2^{n-1} \gamma) \cdot y : \gamma \in S_y\}}
\sum_{[z]_{(x,h_{n-1})}} f \right).  
\end{multline*}
where $S_y = \{\gamma \colon \norm{\gamma}_\infty = 1 \land (2^{n-1}
\gamma) \cdot y \in [y]_{(x,h_n)}\}$.  Note that $\{[(2^{n-1} \gamma) \cdot
y]_{(x,h_{n-1})} : \gamma \in (S_y \union \{0\})\} = Q_{x,h_n}(y)$ and so
$S_y$ has $2^d - 1$ elements since $Q_{x,h}(y)$ has $2^d$ many elements. 

Using our induction hypothesis that $\theta_{n-1}(y) = \frac{1}{2^{(n-1)d}}
\sum_{[y]_{(x,h_{n-1})}} f$ and simplifying, we get:
\[
\theta_n(y) = \frac{1}{2^{nd}} \left( \sum_{[y]_{(x,h_{n-1})}} f 
+ \sum_{z \in \{(2^{n-1} \gamma)
\cdot y : \gamma \in S_y\}}
\sum_{[z]_{(x,h_{n-1})}} f \right)\]
which is equal to
$ \frac{1}{2^{nd}} \sum_{[y]_{(x,h_n)}} f$
using the fact from above that
$\{[(2^{n-1} \gamma) \cdot y]_{(x,h_{n-1})} : \gamma \in (S_y
\union \{0\})\} = Q_{x,h_n}(y)$ is a partition of $[y]_{(x,h_n)}$. 
\end{proof}

Lemma~\ref{error_lemma} implies that for every $x \in X$ and $(h_0,
h_1 \ldots) \in \hat{\Z^d}$, the function $\lim_{n \to \infty} \psi_{x,(h_0,
\ldots, h_n)}(y,z)$ will be an $f$-flow provided it converges everywhere
and $\lim_{n \to \infty} \frac{1}{2^{nd}} \sum_{[y]_{(x,h_n)}} f \to 0$
everywhere. However, if $a \from \Z^d \actson X$ is a Borel action, then we
cannot hope to pick a single point $x$ out of each
connected component of $G_a$ in a Borel way to use as a base point in this
construction. (For example, for the action $a_{\bfu}$, a set which meets
$\T^k$ exactly once in each orbit must be nonmeasurable).
Instead, we will average the above construction over every possible element
of $\hat{\Z^d}$, and then use the fact that the resulting function does not
depend on the base point $x$ that we choose. 

For every $x \in X$, define
\[\psi_x(y,z) = \sum_{n > 0} \frac{1}{2^{nd}} \sum_{h \in \Z^d/(2^n
\Z)^d} \phi_{x,h}(y,z) - \phi_{x,h}(z,y).\]
Finally, let $\psi$ be defined on every edge $(y,z)$ in $G_a$ by 
\[\psi(y,z) = \psi_y(y,z).\]

\begin{lemma}\label{flow_lemma}
  Suppose there is a function $\Phi \from \N \to \R$ such that for every $y \in X$
  \[\left|\sum_{R_{2^n} \cdot y} f \right| < \Phi(2^n)\]
  and
  \[c = \frac{1}{2^{d-1}} \sum_{n = 0}^\infty
  \frac{\Phi(2^n)}{2^{n(d-1)}}\] 
  is finite. Then $\psi$ is an $f$-flow bounded by $c$. 
\end{lemma}
\begin{proof}
  We begin by showing that for every $x \in X$, $\psi_x$ is an $f$-flow of
  $G \restriction [x]_{G_a}$ bounded by $c$.

  If $h_n \in \Z^d/(2^n \Z)^d$, then since 
  $n_{x,h_n}(y,s) \leq 2^{n-1}$,
  and $|\sum_{Q_{x,h}(y,\gamma)} f)| \leq \Phi(2^{n-1})$, we get 
  $|\phi_{x,h}(y,\gamma \cdot y)| \leq 2^{n-1} \frac{\Phi(2^{n-1})}{2^{nd}}$ 
  Hence, 
  \[|\psi_x(y,z)| \leq \sum_{n = 1}^\infty 2 \cdot 2^{n-1}
  \frac{\Phi(2^{n-1})}{2^{nd}} = 
  \frac{1}{2^{d-1}} \sum_{n = 1}^{\infty}
  \frac{\Phi(2^{n-1})}{2^{(n-1)(d-1)}} = c.
  \]

  If we consider the first $n$ terms in
  the summation defining $\psi_x$, this is equal to 
  the average of $\psi_{x,(h_0, \ldots, h_n)}$
  over all sequences $(h_0, \ldots, h_n)$ with $h_i \in \Z^d/(2^i \Z)^d$
  and $\pi_i(h_i) = h_{i-1}$. This is because there are the same number of
  these sequences
  $(h_0, \ldots, h_n)$ containing any given $h_i \in \Z^d/(2^i \Z)^d$. 
  Each such $\psi_{x,(h_0, \ldots, h_n)}$
  is an $(\epsilon,f)$-flow for $\epsilon = \Phi(2^n)/2^{nd}$ by
  Lemma~\ref{error_lemma}.
  Since the average of finitely many
  $(\epsilon,f)$-flows is an $(\epsilon,f)$-flow, it follows that $\psi_x$
  is a limit of $(\epsilon,f)$-flows with $\epsilon = \Phi(2^n)/2^{nd}$,
  which approaches $0$ as $n$ goes to $\infty$. Finally, $\psi_x(y,z) = -
  \psi_x(z,y)$ by definition. This finishes the proof that $\psi_x$ is an
  $f$-flow bounded by $c$.

  To show that $\psi$ is an $f$-flow of $G_a$, it is enough to show that
  for all $x \in X$ and $g \in \Z^d$, $\psi_x = \psi_{g \cdot x}$.
  Now $P_{g \cdot x,-g+ h} = P_{x,h}$ and so
  $\phi_{g \cdot x,-g+ h} = \phi_{x,h}$. From this we can
conclude that $\psi_{g \cdot x} = \psi_x$, since each
term in the summation defining $\psi_x$
averages $\phi_{x,h}$ over all $h \in \Z^d/(2^n\Z)^d$, which is
equal to the average of $\phi_{g \cdot x,-g + h}$ over all $h \in
\Z^d/(2^n\Z)^d$, and hence the average of $\phi_{g \cdot x, h}$ over all $h
\in \Z^d/(2^n\Z)^d$.

\end{proof}

\section{Integral Borel flows}
\label{integral_section}

Suppose $G$ is a graph on $X$, and $f \from X \to \Z$ is a function. 
An \define{integral} $f$-flow is an $f$-flow $\phi$
so that $\phi(x,y)$ is an integer for every edge $(x,y)$ in $G$. 
In this section, we consider the problem of turning real-valued $f$-flows into
integral $f$-flows. 

Classically, the question of when a locally finite graph admits an integral $f$-flow
is easy to answer. It is usually called the \define{integral flow theorem}.

\begin{thm}[The integral flow theorem \cite{D}]\label{ift}
  Suppose $G$ is a locally finite graph on $X$, $c$ is a capacity function
  for $G$ that takes integer values, and $f \from X \to \Z$ also takes
  integer values. If there is an $f$-flow bounded by $c$, then there
  is an integral $f$-flow bounded by $c$.
\end{thm}
\begin{proof}
  Suppose first that $G$ is a finite graph. Consider the graph $G'$ defined
  in the proof of Theorem~\ref{finite_flow_existence}. If we use the
  Ford-Fulkerson algorithm (see \cite{D}) to find a maximal flow for $G'$,
  it will make a flow with only integer values. 

  For infinite graphs, the theorem follows from the finite case using the
  same idea as the proof of Theorem~\ref{infinite_flow_existence}. This is
  because an ultralimit of integer valued functions is integer valued. 
\end{proof}

We also have the following folklore theorem which shows that if $f$ takes integer
values, then we can find an integral $f$-flow ``close'' to any real-valued
$f$-flow. 

\begin{cor}\label{bounded_int_flow}
  Suppose $G$ is locally finite graph on $X$ and $f \from
  X \to \Z$ takes integer values. If $\phi$ is an $f$-flow, then there is
  an integral $f$-flow $\psi$ such that 
  \[|\phi(x,y) - \psi(x,y)| < 1\]
  for every edge $(x,y)$ in $G$.
\end{cor}
\begin{proof}
  Let 
  \[\phi'(x,y) = \begin{cases} \lfloor \phi(x,y) \rfloor & \text{ if
  $\phi(x,y) \geq 0$} \\
  \lceil \phi(x,y) \rceil & \text{ if $\phi(x,y) < 0$}
  \end{cases}\]
  Then $\phi'(x,y) = - \phi'(x,y)$, and $\phi'$ is an integral $f'$-flow
  for the function $f'(x) = \sum_{y \in N(x)} \phi'(x,y)$. 
  Thus, $(\phi - \phi')$ is an $(f - f')$-flow which is bounded by the
  capacity function where $c(x,y) = 0$ if $\phi(x,y) - \phi'(x,y) \leq 0$ and
  $c(x,y) = 1$ otherwise. Now applying 
  Theorem~\ref{ift}, there is an integral $(f - f')$-flow $\phi''$ which
  is bounded by $c$. Finally, adding again we see that
  $\psi = \phi'' + \phi'$ is an integral $f$-flow, and $|\phi(x,y) -
  \psi(x,y)| < 1$.
\end{proof}

Note that since $|\phi(x,y)- \psi(x,y)| < 1$ in
Corollary~\ref{bounded_int_flow}, if $\phi(x,y)$ is already an
integer, then $\psi(x,y) = \phi(x,y)$. 

Now in the Borel setting, we have the following corollary of \cite{L88}
which implies that we cannot always turn a real valued Borel $f$-flow into
an integral Borel $f$-flow. 

\begin{cor}[\cite{L88}]
  There is a $2$-regular acyclic Borel graph $G$ on $X$ and a Borel function $f
  \from X \to \Z$ such that $G$ has a Borel $f$-flow, but does not have an
  integral Borel $f$-flow.
\end{cor}
\begin{proof}
  Let $G$ be the graph defined by Laczkovich in \cite{L88}. This graph is a $2$-regular acyclic Borel graph $G$ on a standard Borel space $X$, that admits a 
  bipartition into two Borel sets $X_0$ and $X_1$. By \cite{L88} this graph
  has no Borel (or even Lebesgue-measurable)
  perfect
  matching.
  Now let $f \from X \to \{-1,1\}$ be the function where $f(x) =
  1$ if $x \in X_0$ and $f(x) = -1$ if $x \in X_1$. If we let
  $\phi(x,y) = 1/2$ and $\phi(y,x) = -1/2$ for every edge $(x,y)$ in $G$ where $x \in X_0$ and $y
  \in X_1$, then $\phi$ is clearly a Borel $f$-flow. 
  
  However, $G$ does not have an integer-valued Borel $f$-flow. For a
  contradiction, suppose $\psi$ was in integer-valued Borel $f$-flow for $G$.
  Then the set of edges $(x,y)$ such that $x \in X_0$ and $y \in X_1$ and
  $\psi(x,y) > 0$ would be a Borel perfect matching of $G$. 
\end{proof}

Despite this, we do have the following ``Borel integral flow theorem'' for graphs induced by free
Borel actions of $\Z^d$ for $d \geq 2$.

\begin{lemma}\label{Borel_integral_flow}
  Suppose $d \geq 2$, $a \from \Z^d \actson X$ is a free Borel action,
  and $G_a$ is the associated graph. Then if $f \from X \to \Z$ is a
  Borel function and $\phi$ is a Borel $f$-flow for $G$, then
  there is an integral Borel $f$-flow $\psi$ such that $|\phi - \psi|
  \leq 3^{d}$.
\end{lemma}

In our proof of Theorem~\ref{main_thm}, we will use
Lemma~\ref{Borel_integral_flow} to turn the real-valued Borel flow
constructed in Lemma~\ref{flow_lemma} into an integral Borel flow.

Our proof of Lemma~\ref{Borel_integral_flow} uses the following result of Gao,
Jackson, Krohne, and Seward (see Appendix~\ref{cake_appendix}). Their theorem
answers a question originally due to Ben Miller. 

Suppose $F$ is a finite set of vertices in a graph $G$, let $\boundary F$
be the set of edges that are incident on one vertex in $F$ and one vertex
not in $F$. Now let $\boundary^1 F = \boundary F$, and let $\boundary^{n+1}
F$ be the set of edges that are in $\boundary^n F$ or share a vertex with
an edge in $\boundary^n F$. Finally, let $\boundaryinfty F \subset
\boundary F$ be the set of edges in the boundary of $F$ that are ``visible
from infinity''. That is, $\boundaryinfty F$ is the set of edges $e \in
\boundary F$ such that the unique $x \notin F$ incident to $e$ 
is such that the connected component of $x$ in $G - \boundary F$ is infinite. 

\begin{thm}[Gao, Jackson, Krohne, and Seward]\label{cake}
  Suppose $d \geq 1$, $n > 0$, and $a \from \Z^d \actson X$ is a free Borel
  action of $\Z^d$ on a standard Borel space $X$, and $G_a$ is the
  associated graph. Then there is a Borel set $C \subset [X]^{< \infty}$ such
  $\bigunion C = X$, for every distinct $R, S \in C$, $\boundary^n R$ and
  $\boundary^n S$ are disjoint, and every $S \in C$ is connected and has
  $\boundary S = \boundaryinfty S$. 
\end{thm}

For an annoucement of this theorem, see the paragraph following Corollary 1.8 in \cite{GJKS}.

The point of this theorem is that the elements of $C$ cover the whole space
$X$, and their boundaries can be chosen to be arbitrarily far apart. The condition that $\boundary
S = \boundaryinfty S$ is incidental and is included just to make
Lemma~\ref{euler_lemma} a little simpler. Given any connected $S \subset
X$, let $\tilde{S}$ be the set of vertices in $S$ together with all $x \in
X$ that are in some finite connected component of $G - \boundary S$.
The idea here is that $\tilde{S}$ is obtained by filling in any ``holes'' inside
$S$. It is clear that $\boundary \tilde{S} \subset \boundary S$, and
that $\boundaryinfty \tilde{S} = \boundary \tilde{S}$. Hence, if $C \subset
[G]^{< \infty}$ satisfies all the conditions of the theorem except the
condition on $\boundaryinfty$, then we can simply replace $C$ with
$\{\tilde{S} : S \in C\}$ to satisfy this last condition.

We need a short combinatorial lemma. Recall than an Euler cycle in a finite
graph is a closed walk that includes each edge in the graph exactly once.

\begin{lemma}\label{euler_lemma}
  Suppose $d \geq 2$, $a \from \Z^d \actson X$ is a free action,
  $G_a$ is the associated graph, and $F \subset X$ is a finite
  $G_a$-connected set with $\boundaryinfty F = \boundary F$. Let
  $H_{\boundary F}$ be the
  graph whose vertex set is the unordered edges in $\boundary F$, that is, $\{\{x,y\} : (x,y) \in \boundary F\}$ and where
  distinct $\{x,y\}, \{z,w\}$ are adjacent in $H_{\boundary F}$ if there is a $3$-cycle 
  in $G_a$ that includes them both. Then $H_{\boundary
  F}$ has an
  Euler cycle.
\end{lemma}
\begin{proof}
  By Euler's theorem, we
  need to show that $H_{\boundary F}$ is connected and every vertex of
  $H_{\boundary F}$ has even degree.

  We begin by showing every edge has even degree.
  Fix $(x,y) \in
  \boundary F$.
  Any $3$-cycle in $G_a$ that contains $(x,y)$ must contain exactly one
  other edge in $\boundary F$. Thus, it suffices to show that
  there are an even number of $3$-cycles in $G_a$ containing the edge $(x,y)$. 
  Let $(x,y) = (x,\gamma \cdot x)$ where $\norm{\gamma}_\infty = 1$.
  Then the number of $3$-cycles containing $(x,y)$ is equal to the number
  of $\delta$ with $\norm{\delta}_\infty = 1$ and $\norm{\gamma - \delta}_\infty =
  1$. If $\gamma = (\gamma_1, \ldots, \gamma_d)$, then this $\delta =
  (\delta_1, \ldots, \delta_d)$ must have $\delta_i \in \{-1,0,1\}$ if
  $\gamma_i = 0$, $\delta_i \in \{0,1\}$ if $\gamma_i = 1$, and $\delta_i
  \in \{-1,0\}$ if $\gamma_i = -1$. Thus, if there are $k$ many values of $i$ such
  that $\gamma_i = 0$, then the number of $\delta$ with this property is
  $3^k 2^{d - k} - 2$, where we subtract $2$ since neither $\delta$ nor
  $\gamma - \delta$ can be equal to $0$. To finish, note that $3^k 2^{d - k} -
  2$ is even since $k < d$. 

  Next, we claim that $H_{\boundary F}$ is connected. To see this will use
  \cite{Ti}. First, we claim that the set of all $3$-cycles in $G_a$
  generates the cycle space of $G_a$. That is, every cycle in $G_a$ is a
  sum of finitely many $3$-cycles, where we add edges in the cycles modulo
  $2$ (see \cite{Ti}). This is easy to see, and we sketch an argument in
  the case where $a$ is the translation action of $\Z^d$ on itself. Let
  $e_i$ be the $i$th element of the usual basis for $\Z^d$. Let $T$ be the
  spanning subtree of $G_a$ where for each $x = (x_1, x_2, \ldots, x_d) \in
  \Z^d$, we have $(x, e_i \cdot x) \in T$ if $x_j = 0$ for all $j > i$.
  Given an edge $(x,y) \in G_a$ such that $(x,y) \notin T$, let $C_{(x,y)}$
  be the unique cycle created by adding $(x,y)$ to $T$. An easy induction
  shows that every such $C_{(x,y)}$ is a sum of $3$-cycles, but clearly any
  cycle in $G_a$ is a sum of cycles of the form $C_{(x,y)}$. 
  
  Since $F$ is connected and $\boundaryinfty F = \boundary F$, if we choose
  any $x \in F$ and $y \notin F$ where $y \in [x]_{G_a}$, then $\boundary
  F$ is a minimal set of edges separating $x$ and $y$ in the sense that any
  proper subset of $\boundary F$ does not separate $x$ and $y$. Hence, by
  \cite[Lemma 1]{Ti}, if $H_{\boundary F}$ had two connected components
  $\Pi_1$ and $\Pi_2$, then there would be some $3$-cycle in $G_a$ that
  intersects both $\Pi_1$ and $\Pi_2$. But this is clearly a contradiction.
  \end{proof}

We are now ready to prove Lemma~\ref{Borel_integral_flow}.

\begin{proof}[Proof of Lemma~\ref{Borel_integral_flow}]
Let $\phi$ be a Borel $f$-flow for $G_a$. Let $C \subset X^{\infty}$ be as
in Theorem~\ref{cake}, with $n = 3$. We use $n = 3$ here for the following
reason: if $R, S \in C$ and we change $\phi$ separately on $\boundary^2 R$
and $\boundary^2 S$ in a such a way that it remains a flow after each
individual modification, then if we combine both modifications the result
will also be a flow. 

We define another Borel $f$-flow $\phi'$ as follows. Let $\phi'(x,y) =
\phi(x,y)$ if $(x,y)$ is not in $\boundary^2 F$ for any $F \in C$. If
$(x,y) \in \boundary^2 F$ for some $F \in C$, then there is a unique such
$F$. Let $((e_1, e'_1, e''_1), \ldots, (e_n,
e'_n, e''_n))$ be the sequence of $3$-cycles in $G_a$ associated to the
lexicographically least Euler cycle of $H_{\boundary F}$, where we are representing each
$3$-cycle by the edges contained in it. An Euler cycle of $H_{\boundary F}$
exists by Lemma~\ref{euler_lemma}.
We may arrange these $3$-cycles so that for all $i$, we have $e_i, e'_i \in
\boundary F$, $e''_i \notin \boundary F$, and $e_i' = e_{i+1}$. Orient the
edges so that $(e_i, e'_i, e''_i) = ((x_i, y_i), (y_i, z_i), (z_i, x_i))$.

Let $\phi^F_0 = \phi$. Given $(e_i, e'_i, e''_i)$, let $\alpha_i =
\phi^F_i(e_i) - \lfloor \phi^F_i(e_i) \rfloor$. Then define
$\phi^F_{i+1}(u,v) = \phi^F_i(u,v)$ if $(u,v)$ is not an edge in the cycle
$(e_i, e'_i, e''_i)$, and otherwise let $\phi^F_{i+1}(u,v) = \phi^F_i(u,v)
- \alpha_i$ if $(u,v)$ is oriented the same direction as the cycle $e_i,
e'_i, e''_i$, and $\phi^F_{i+1}(u,v) = \phi^F_i(u,v) + \alpha_i$ if $(u,v)$
is oriented in the opposite direction. Hence, $\phi^F_{i+1}$ is still an
$f$-flow since we are modifying $\psi^F_i$ only by adding the same amount
to each edge going around a single cycle. Finally, let $\phi'(x,y) =
\phi^F_n(x,y)$. 
  
We claim that if $e \in \boundary F$, then $\phi'(e)$ will be an integer.
First suppose that $e \neq e_n'$. Then if $j$ is the largest number such
that $e = e_j$ (up to direction), we define $\phi^F_j(e)$ to be an integer,
and $e$ cannot equal $e_k'$ for any $k > i$ since $e_k' = e_{k+1}$, hence
$\phi^F_k(e) = \phi^F_j(e)$ for all $k > i$. Thus $\phi'(e)$ is an integer. If $e = e_n'$, then since $f$
is integer valued, the total flow out of $F$ must be an integer. So since
$\phi'$ takes integer values on all the other edges in $\boundary F$,
$\phi'(e)$ must also be an integer.

Note that $|\phi'(x,y) - \phi(x,y)| \leq 3^d-1$ for every edge $(x,y)$.
This is because every node in $H_{\boundary F}$ has degree at most $3^d -
1$, and we change the flow on each edge by at most $1$ as we go around the
Euler cycle. 

Now let $D$ be the set of edges $(x,y)$ in $G_a$ such that $(x,y) \in
\boundary F$ for some $F \in C$. $G_a - D$ has finite connected components,
since $\bigunion C = X$. Let $K$ be a connected component of $G_a - D$. If
$\theta$ is an integer-valued function defined on the edges of $K$, let
$\phi'_{\theta}(x,y) = \theta(x,y)$ if the edge $(x,y)$ is in $K$, and
$\phi'(x,y)$ otherwise. By Corollary~\ref{bounded_int_flow}, there is a
integer valued function $\theta$ defined on the edges of $K$ so that
$\phi'_{\theta}(x,y)$ is an $f$-flow, and $|\phi'_{\theta} - \phi'| < 1$. 

Let $\psi$ be the Borel $f$-flow on $G_a$ defined as follows. Let $\psi(x,y) =
\phi'(x,y)$ if $(x,y) \in D$. If $(x,y) \notin D$, then let $K$ be the
connected component of $G_a - D$ containing $(x,y)$, let $\theta$ be the
lexicographically least integer-valued function on the edges of $K$ such that
$\phi'_{\theta}$ is an $f$-flow where $|\phi'_{\theta} - \phi'| < 1$,
and let $\psi(x,y) = \phi'_{\theta}(x,y)$.
\end{proof}

\section{Proof of Theorem~\ref{main_thm}}\label{main_proof_section}

In order to prove Theorem~\ref{main_thm} using the fewest number of pieces in our
equidecomposition, we use the following lemma due to Gao and Jackson, which
was an important ingredient in their proof that Borel actions of abelian
groups are hyperfinite. In
Remark~\ref{voronoi}, we describe how one can prove Theorem~\ref{main_thm}
without using this black box. 

\begin{lemma}[\cite{GJ}]\label{rect_tiling}
  Suppose $a \from \Z^d \actson X$ is a free Borel action of $\Z^d$ on a standard
  Borel space $X$ and $n > 0$. Then
  there is a Borel set $C \subset [X]^{< \infty}$ such that $C$ partitions
  $X$ and every $S \in
  C$ is a set of the form $\{(n_1,\ldots, n_d) \cdot x : 0 \leq
  n_i < N_i\}$ for some $x \in X$ and sequence $N_1, \ldots, N_d$ where
  $N_i = n$ or $N_i = n+1$.
\end{lemma}

Roughly, the above lemma states that there is a Borel tiling of the action
$a$ using rectangles each of whose side lengths is $n$ or $n+1$.

\begin{proof}[Proof of Theorem~\ref{main_thm}.]
  As discussed at the beginning of Section~\ref{sec:defn}, we may assume
  that $A, B \subset \T^k$.
  Let $d \geq 2$ be such that $d > 2k/(k - \Delta(\boundary A))$ and pick
  $\bfu \in (T^k)^d$ such that the action $a_{\bfu}$ is free and satisfies
  Lemma~\ref{L_discrep} for both sets $A$ and $B$. Hence, there is some $M$ and
  $\epsilon > 0$ such that 
  \[D(F_{N}(x,a_{\bfu}),A) \leq M N^{-1 - \epsilon}\quad \text{ and } \quad
  D(F_{N}(x,a_{\bfu}),B) \leq M N^{-1 - \epsilon}\]
  for every $x \in \T^k$ and $N > 0$.

  Now consider the graph $G_{a_\bfu}$ and the function $f = \chi_A -
  \chi_B$, the difference between the characteristic functions of $A$ and
  $B$. Since $\lambda(A) = \lambda(B)$, by the definition of discrepancy
  and $f$, for every $x \in \T^k$ and $n > 0$, 
  \[\left | \sum_{F_{2^n}(x,a_{\bfu})} f \right| = 2^{nd} \left|
  D(F_{2^n}(x,a_{\bfu}),A) - D(F_{2^n}(x,a_{\bfu}),B) \right| \leq 2M 2^{n(d - 1 -
  \epsilon)}.
  \]
  Thus, by Lemma~\ref{flow_lemma}, letting $\Phi(2^n) = M 2^{n(d - 1 -
  \epsilon)+1}$, there is a bounded Borel $f$-flow for the graph
  $G_{a_{\bfu}}$. 
  By Lemma~\ref{Borel_integral_flow}, there is a bounded integral Borel
  $f$-flow for the graph $G_{a_\bfu}$. Call this integral Borel flow $\psi$, and
  suppose that $\psi$ is bounded by the constant $c$. 
  
  For each $n$, by Lemma~\ref{rect_tiling}, let $C_n \subset (\T^k)^{<
  \infty}$ be a Borel tiling of the action $a_\bfu$ by rectangles of side lengths
  $n$ or $n+1$. For each $x \in \T^k$, let $V_n(x)$ be the unique element of $C_n$
  that contains $x$. Now for every $x \in \T^k$, $|\boundary V_n(x)| \leq 2d \cdot
  3^d \cdot (n+1)^{d-1}$
  which is $O(n^{d-1})$.
  Next, since there is some $x' \in V_n(x)$ such that $F_n(x',a_\bfu)
  \subset V_n(x)$, we have that
  $|A \inters V_n(x)| \geq \lambda(A)n^d - M n^{d - 1 - \epsilon}$. 
  Hence, there is some $K$ so that 
  \[c |\boundary V_K(x)| \leq |A \inters V_K(x)| \quad \text{ and } \quad
  c |\boundary V_K(x)| \leq |B \inters V_K(x)| \]
  for every $x \in \T^k$. Fix this $K$ and let $C = C_K$.

  For each $R \in C$, let $N(R)$ be the set of $S \in C$ such that $S \neq
  R$ and $\boundary S \inters \boundary R \neq \emptyset$. Note that $N(R)$ is finite.
  Given $S \in N(R)$, let 
  \[\Psi(R,S) = \sum_{\{(x,y) : x \in R \land y \in S\}} \psi(x,y) \]
  so $\Psi(R,S)$ is integer-valued, $\Psi(R,S) = - \Psi(S,R)$ and
  \[\sum_{S \in N(R)} \Psi(R,S) = |R \inters A| - |R \inters B|.\]
  Note also that $\sum_{S \in N(R)} |\Psi(R,S)|$ is 
  less than $|R \inters A|$ and $|R \inters B|$ by our choice of $K$ and if 
  $S \in
  N(R)$, then for any $x \in R$ and
  $y \in S$, there is a $\gamma \in \Z^d$ with $\norm{\gamma}_\infty < 2K + 4$ such that
  $\gamma \cdot x = y$. 

  Essentially, if we let $G_C$ be the graph with vertex set $C$ where $R$
  is adjacent to $S$ if $\boundary R \inters \boundary S$, then $\Psi$ is a
  flow on this graph for the function $f(R) = \sum_{R} \chi_A - \chi_B$,
  and $N$ is the neighborhood relation on this graph. 

  To show that $A$ and $B$ are $a_{\bfu}$-equidecomposable
  using Borel pieces, it suffices to construct a Borel bijection $g \from A
  \to B$ such that for all $x \in A$, $g(x) = \gamma \cdot x$ for some
  $\gamma$ such that $\norm{\gamma}_\infty < 2K + 4$. Then the pieces in
  our equidecomposition will be $\{x \in A : g(a) = \gamma \cdot x\}$ for
  each $\gamma$ with $\norm{\gamma}_\infty < 2K + 4$. 

  Our idea for constructing $g$ is that for
  every $R, S \in C$, if $\Psi(R,S) > 0$ we should map $\Psi(R,S)$ many
  points from $A \inters R$ to point of $B \inters S$. After doing this for
  all pairs $R,S$, there will be an equal number of points of $A$ and $B$
  left in each $R \in C$, so we map the remaining points in $A \inters R$
  to the points of $B \inters R$. 
  
  Fix a Borel linear ordering $<_C$ of $C$, and a Borel linear ordering $<$
  of $\T^k$. For each $R \in C$, inductively let $A(R,S) \subset A \inters
  R$ be the least $\Psi(R,S)$ many elements of $A \inters R$ that are
  not in $A(R,S')$ for any $S' \in N(R)$ where $S' <_C S$. Similarly, for
  each $R \in C$, let $B(R,S) \subset B \inters R$ be the first
  $\Psi(R,S)$ many elements of $B \inters R$ that are not in $B(R,S')$ for
  any $S' \in N(R)$ with $S' <_C S$. Finally, let $A'(R) = A \inters R
  \setminus \bigunion_{\{S \in N(R) : \Psi(R,S) > 0\}}A(R,S)\}$ and $B'(R)
  = B \inters R \setminus \bigunion_{\{S \in N(R) : \Psi(R,S)\}}B(R,S)$. By
  the properties of $\Psi$ listed above, $|A'(R)| = |B'(R)|$ for every $R
  \in C$. Define $g \from A \to B$ as follows. Given $x \in A$, let $R$ be
  the unique element of $C$ containing $x$.
  If there is some $S \in N(R)$ such that $x \in A(R,S)$, and $x$ is
  the $l$th-least element of $A(R,S)$, then let $g(x)$ be the $l$th-least
  element of $B(S,R)$. If not, then $x \in A'(R)$. If $x$ is the
  $l$th-least element of $A'(R)$, then let $g(x)$ be the $l$th-least
  element of $B'(R)$.
\end{proof}

\begin{remark}\label{voronoi}
  We sketch an alternate proof of Theorem~\ref{main_thm} without using
  Lemma~\ref{rect_tiling}.
  By Theorem~\ref{max_indep}, for each $n$, there
  is a Borel maximal $n$-discrete set $C_n$ for $G_{a_{\bfu}}$. Note that
  the $n/2$-balls around points in $C_n$ are pairwise disjoint. 
  Given $x
  \in C_n$, let $V_n(x)$ be the \define{Voronoi cell} determined by the seed $x$ in
  the graph $G_{a_\bfu}$. That is, $V_n(x)$ is the set of $y \in \T^k$ such that
  $x$ is the $<$-least element of $V_n$ such that $d(y,x) \leq d(y,z)$ for
  all $z \in \T^k$. Note that $\{V_n(x) : x
  \in C_n\}$ is a Borel partition of $\T^k$.
  Since the $n/2$-balls around points in $C_n$ are pairwise disjoint, every
  set $V_n(x)$ contains a set of the form $F_n(x',a_\bfu)$ (where $x'$ is a point of
  distance $n/2$ from $x$).

  Next, we compute an upper bound on the size of $\boundary V_{n}(x)$.
  Fix $x \in X$. The $3n$-ball around $x$ has size $(6n+1)^d \leq (7n)^d$.
  Since the $n/2$-balls around points in $B$ are disjoint and have size $(n
  + 1)^d \geq n^d$, there are at most $(7n)^d / n^d = 7^d$ many points $y$ of
  distance $\leq 2n$ from $x$, since the $n/2$-ball around $y$ must be
  contained in the $3n$-ball around $x$. Now given any two points $x, y \in
  B$ such that $d(x,y) \leq 2n$, the set of $z$ such that $d(x,z) = d(y,z)
  \leq n$ has size $O(n^{(d-1)})$. Since there are at most $7^d$ such $y
  \in B$ in the $2n$-ball around $x$, the boundary of $V_{n}(x)$ has size
  $O(n^{(d-1)})$. Thus, we can find some $K$ such that for every $x \in X$,
  \[ c |\boundary V_{K}(x)| \leq  |V_{K}(x) \inters A| \quad \text{ and } \quad 
  c |\boundary V_{K}(x)| \leq |V_{K}(x) \inters B| \]
  Now finish as before using these Voronoi cells $\{V_{K}(x) : x \in C_K\}$
  instead of the Gao-Jackson tiling.
\end{remark}

\section{The Borel complexity of our equidecompositions}
\label{complexity_section}

In this section, we make some remarks about the complexity of the
Borel pieces used in the proof of Theorem~\ref{main_thm}. The task of
computing these complexities is standard, and we merely sketch a outline of
the argument to show the pieces are $\boolH^{A,B}_4$.  Fix $A, B \subset
\T^k$ and the action $a_{\bfu} \from \Z^d \actson \T^k$ from the proof of
Theorem~\ref{main_thm}. We begin with a remark we will use several times
when computing complexities.

\begin{remark}\label{complexity_remark}
Suppose $C \subset \T^k$ is defined in terms of some sets $D_1, \ldots, D_n
\subset \T^k$. If there is some $m$ and a deterministic algorithm which
decides if $x \in C$ based on inspecting what vertices of the $m$-ball
around $x$ in the graph $G_{a_\bfu}$ lie in $D_1, \ldots, D_n$, then 
$C$ is a finite boolean combination of the sets $g \cdot D_i$
where $\norm{g}_\infty \leq k$ and $1 \leq i \leq n$.
Hence, if $D_1, \ldots, D_n \in \boolH^{A,B}_m$, then we also
have $C \in \boolH^{A,B}_m$. 
\end{remark}

In the cases where we apply Remark~\ref{complexity_remark} to compute
complexities, such an algorithm will be clear from the proof where we construct
the corresponding set. 

We now proceed to calculate the complexity of the sets at each step in our
argument. Recall from Section~\ref{flow_section} that 
\[\psi(x,\gamma \cdot x) = \sum_{n > 0} \frac{1}{2^{nd}} \sum_{h \in \Z^d/(2^n
\Z)^d} \phi_{x,h}(x,\gamma \cdot x) - \phi_{x,h}(x,\gamma \cdot x).\]
Let $\psi_k(x,\gamma \cdot x)$ be the first $k$ many terms of this
summation, so 
\[\psi_k(x,\gamma \cdot x) = \sum^k_{n = 1} \frac{1}{2^{nd}} \sum_{h \in \Z^d/(2^n
\Z)^d} \phi_{x,h}(x,\gamma \cdot x) - \phi_{x,h}(x,\gamma \cdot x).\]
Then it is clear that $\psi_k$ can take only finitely many 
rational values, and for each $\gamma \in \Z^d$ with $|\gamma|_\infty = 1$,
and every possible value $a$, by Remark~\ref{complexity_remark},
\[\{x \colon \psi_k(x,\gamma \cdot x) = a\} \in \boolH^{A,B}_1.\]
Define 
\[\epsilon_k = \frac{1}{2^{d-1}} \sum_{n = k-1}^\infty
\frac{\Phi(2^n)}{2^{n(d-1)}}\] 
to be the tails of the summation defining
the constant $c$ in Lemma~\ref{flow_lemma}, where $\Phi$ is as in the
proof of Theorem~\ref{main_thm}. Then $|\psi_k(x,y) - \psi(x,y)| <
\epsilon_k$ for every edge $(x,y)$ in $G_{a_{\bfu}}$. Hence,
for each real number $a$,  
\[\{x \colon \psi(x,\gamma \cdot x) < a\} \in \bfSigma^{A,B}_2\]
since $x$ is in this set if and only if $\psi_k(x,\gamma \cdot x) +
\epsilon_k < a$ for some $k$. 
Indeed, for any finite sequence $(g_1, \gamma_1), \ldots, (g_n,\gamma_n)$,
where $g_i, \gamma_i \in \Z^d$ and $|\gamma_i|_\infty = 1$, we have that
\[\left\{x \colon \left(\sum_{i = 1}^n \psi(g_i \cdot x,\gamma_i \cdot (g_i
\cdot x)) \right) < a\right\} \in \bfSigma^{A,B}_2\]
by the same argument.

In Section~\ref{sec:defn}, we discussed how at some points in our proof, we
use a Borel linear ordering on $\T^k$ to make an arbitrary choice
in our construction. In order to obtain pieces in our decomposition with
the lowest possible complexity, we need to use a particular ordering.
Since the comparisons we make in our proof are only between points
in the same orbit of $a_{\bfu}$, it is enough to have a Borel partial order
which is linear on each orbit. So given $x, g \cdot x$ in the same orbit of
$a_{\bfu}$, define $x < g \cdot x$ if $g$ is greater than $0$ in the
lexicographic ordering on elements of $\Z^d$. We use this ordering because
it combines well with Remark~\ref{complexity_remark}.

Next, we consider the proof of Theorem~\ref{cake} in
Appendix~\ref{cake_appendix}. The maximal $r_i$-discrete sets constructed
after our statement of Theorem~\ref{max_indep} are $\boolH^{A,B}_1$. It is
straightforward to see that the sets $C_i$ constructed by
Lemma~\ref{accumulation_lemma} are each $\boolH^{A,B}_3$; the 
complicated part of their definition is finding the least $g$ such that
$D^*(x) \inters (g + [0,1]^d)$ is nonempty. 
Hence, for the sets $D_i$ constructed in the proof of Theorem~\ref{cake},
for any $g_1, \ldots, g_n \in \Z^d$,
\[\{x \in \T^k : \{g_1 \cdot x, \ldots, g_n \cdot x\} \in D_i\} \in
\boolH^{A,B}_3 \]
by Remark~\ref{complexity_remark}. Note that for each $i$, there is an upper bound on the
size of all elements of $D_i$. 

Now consider Lemma~\ref{Borel_integral_flow} where we turn our real-valued
flow $\psi$ into an integral Borel flow, which we will call $\psi'$. (Our
variable choices here differ from the $\phi$ and $\psi$ in the statement of
Lemma~\ref{Borel_integral_flow}). We claim that for each $i$, letting $D_i$ be as in the proof of
Theorem~\ref{cake} as above, for every $m$,
\[\{x \in \T^k: x \in \bigunion D_i \land \psi'(x,\gamma \cdot x) = m\} \in
\boolH^{A,B}_3.\]
This is by combining the fact that for each $i$, the elements of $D_i$ 
have bounded size with the fact proved above that
$\left\{x \colon \left(\sum_{i = 1}^n \psi(g_i \cdot x,\gamma_i \cdot (g_i
\cdot x)) \right) < a\right\} \in \bfSigma^{A,B}_2$, 
and then using Remark~\ref{complexity_remark}.
Thus, taking the union of these $\boolH^{A,B}_3$ sets, we see that for
every $m$,
\[\{x \in \T^k: \psi'(x,\gamma \cdot x) = m\} \in
\bfSigma^{A,B}_4.\]
in our integral Borel flow $\psi'$.

Finally, consider the argument in Section~\ref{main_proof_section} which
uses Remark~\ref{voronoi} to define the equidecomposition. Since our
equidecomposition just uses the value of $\psi'$ (whose complexity has been
computed above), and a maximal $K$-discrete set which is $\boolH^{A,B}_1$,
the resulting pieces will be $\boolH^{A,B}_4$ by
Remark~\ref{complexity_remark}.
(Similarly, inspecting the proof of Lemma~\ref{rect_tiling} in \cite{GJ}
also yields pieces with the same complexity at this last step). 

\begin{appendices}
\section{Proof of Theorem~\ref{cake}}
\label{cake_appendix}

In this appendix, we give a proof of Theorem~\ref{cake}. If $d$ is a metric
on a set $X$, then we say that $Y \subset X$ is \define{$r$-discrete (with
respect to $d$)} if $d(x,y) > r$ for all distinct $x, y \subset Y$. We
further say that $Y \subset X$ is a \define{maximal $r$-discrete} set if
$Y$ is $r$-discrete and for every $x \in X$ there is a $y \in Y$ with
$d(x,y) \leq r$. If $G$ is a graph on $X$ then by an $r$-discrete set of
vertices in $X$, we mean with respect to the graph metric $G$ induces on
$X$. 

We will need the following theorem of Kechris, Solecki, and Todorcevic.
\begin{thm}[{\cite[Theorem 4.2]{KST}}]\label{max_indep}
  Every locally finite Borel graph $G$ has a maximal $r$-discrete Borel
  set. 
\end{thm}

In the specific case of the graph $G_{a_{\bfu}}$ we can give a short proof
of this fact. Given any $r >
0$, for sufficiently small $\epsilon > 0$, the $\epsilon$-ball
around any point $x \in \T^k$ will be an $r$-discrete
Borel set. Thus, we can find a maximal $r$-discrete Borel set $C$ for
$G_{a_{\bfu}}$ as follows. Let $B_0, B_1, \ldots, B_{n}$ be finitely many
$\epsilon$-balls which cover $\T^k$. Define $C_{0} = B_0$ and 
let $C_{i+1} = B_{i+1} \setminus \bigunion \{g \cdot C_j : j \leq i \land
|g|_\infty \leq r\}$. So inductively, $\bigunion_{i \leq k} C_i$ is an
$r$-discrete set
and for every $x \in B_k$, there is a $y \in C_k$ such that $d(x,y) \leq
r$. Now finish by letting $C = \bigunion_{i \leq n} C_i$. 

Next, we need the following lemma which is a rephrasing of an idea
due to Boykin and Jackson. They used it to give a new
proof of the theorem originally due to Weiss that free Borel actions of
$\Z^d$ are hyperfinite. 

\begin{lemma}[\cite{BJ}]\label{accumulation_lemma}
  Suppose $d \geq 1$, $n > 0$, $a \from \Z^d \actson X$ is a free Borel
  action of $\Z^d$ on a standard Borel space $X$, and 
  $r_0 < r_1 < \ldots$ is an increasing sequence of natural numbers. Then
  there is a sequence $C_{0}, C_{1}, \ldots \subset X$ of Borel sets such
  that $C_{i}$ is Borel maximal $r_i$-discrete set for $G_a$, and 
  for every $\epsilon > 0$ and every $x \in X$, there are infinitely many
  $i$ such that $d(x,C_i) < \epsilon r_i$.
\end{lemma}
\begin{proof}
For each $i$, by Theorem~\ref{max_indep},
let $C'_{i}$ be a Borel maximal
$r_i$-discrete set for the graph $G_a$. To each $x \in X$, we associate a sequence $D_{0}(x), D_{1}(x) \ldots$
of subsets of $\R^d$. Let
\[D_{i}(x) = \{g/r_i : g \in \Z^d \land (-g) \cdot x \in C'_i\}\] 
So each set $D_{i}(x)$ is a maximal $1$-discrete subset of $\R^d$, with
respect to metric $\norm{ \cdot}_\infty$ on $\R^d$. Now let
$D^*(x) \subset \R^d$ be the set of accumulation points of the sequence
$D_0(x), D_1(x), \ldots$, so $D^*(x)$ is a closed subset of $\R^d$.
$D^*(x)$ is nonempty since each set $D_{i}(x)$ contains at least one
point in the set $[0,1]^d$ which is compact. Now if $x, y \in X$ are in the
same orbit of $a$, then $D^*(x) = D^*(y)$, since the set $D_{i}(x)$ can
be shifted by a distance of $d(x,y)/r_i$ to become equal to $D_{i}(y)$, and
$d(x,y)/r_i \to 0$ as $i \to \infty$. 

Recall that a function $f$ on $X$ is said to be
\define{$a$-invariant} if $f(x) = f(y)$ for all $x, y$ in the same
$a$-orbit of $X$. 
 We claim that there is an
$a$-invariant Borel function $f \from X \to \R^d$ such that $f(x)
\in D^*(x)$ for every $x \in T^k$. 
Let $f_n(x)$ be the lexicographically
least $g \in \{0, \ldots, r_n - 1\}^d$ such that $D^*(x) \inters (g +
[0,1]^d)/r_n$ is nonempty. Then each function $f_n$ is $a$-invariant, and
by the definition of the lexicographic order, the lexicographically least
element of $D^*(x) \inters [0,1]^d$ is contained in $(f_n(x) + [0,1]^d)/r_n$. 
If we let $f(x) = \lim_{n \to \infty} f_n(x)/r_i$, then $f(x)$ is the
lexicographically least element of $D^*(x) \inters [0,1]^d$, and  
$\norm{f(x) - f_i(x)}_\infty < 1/r_i$ for every $i$. Note that $f$ is
$a$-invariant. 

For every $x \in X$ and $\epsilon > 0$ there are infinitely many $i$
such that $d(f(x),D_i(x)) < \epsilon/2$, since $f(x)$ is an accumulation
point of the $D_i(x)$. Hence, there are infinitely many
$i$ such that $d(f_i(x)/r_i, D_i(x)) < \epsilon /2 + 1/r_i$. Using the
definition of $D_i(x)$, this means there are infinitely many $i$ such that
there exists some $y \in C'_i$ such that 
$d((-f_i(x)) \cdot x, y) < \epsilon r_i / 2 + 1$. 

Now if we define
\[C_{i} = \{f_i(y) \cdot y : y \in C'_i\},\] 
then using the $a$-invariance of $f_i$, we see that there are infinitely
many $i$ such that $d(x,C_i) <
\epsilon r_i/2 + 1$ which suffices to prove the theorem.
\end{proof}

It seems likely that one can prove an analogous result for any finitely
generated nilpotent group $\Gamma$ instead of $\Z^d$ by using the Mal'cev
completion of $\Gamma$ in place of $\R^d$ in the argument above.
However, the analogous problem for arbitrary finitely
generated amenable groups is open. A positive answer would imply that every free
Borel action of a finitely generated amenable group is hyperfinite, which
is a well-known open problem (see \cite{JKL}, \cite{GJ}, and \cite{SS}).

\begin{problem}
  Let $\Gamma$ be a finitely generated amenable group with symmetric
  generating set $S$, and $a \from \Gamma \actson X$ be a free Borel
  action of $\Gamma$ on a standard Borel space $X$. Let $G_{a,S}$ be the
  Borel graph on $X$ where there is an edge between $x$ and $y$ if there
  exists a $\gamma \in S$ such that $\gamma \cdot x = y$. 
  Must it be the case
  that for every increasing sequence $r_0 < r_1 < \ldots$, there is a
  sequence $C_{0}, C_1, \ldots$ of Borel subsets of $X$ such that $C_i$ is
  a Borel maximal $r_i$-discrete set, and for every $\epsilon > 0$ and
  $x \in X$, there are infinitely many $i$ such that $d_{G_{a,S}}(x,C_i) <
  \epsilon r_i$?
\end{problem}

Lemma~\ref{accumulation_lemma} will be used in our proof of
Lemma~\ref{cake} to ensure that the elements of $C$ cover $X$. The next
definition and lemma will be used to ensure that elements of $C$ have
disjoint boundaries. 

\begin{defn}
Suppose $S$ is a set of vertices in a graph $G$ on $X$. Then let $B_{r}(S)
= \{x : d(x,S) \leq r\}$ be the ``ball'' of radius $r$ around $S$. Abusing
notation, if $Y \subset [G]^{< \infty}$, let $B_r(Y) = \{B_r(S) : S \in Y\}$.
Finally, if $Y, Z \subset [G]^{< \infty}$, let 
\[ B_{r}(Y,Z) = \left\{S \union \bigunion \{B_r(R) : R \in Z \land d(R,S)
\leq r \}: S \in Y\right\}.\] 
That is, for each $S \in Y$, $B_{r}(Y,Z)$ contains the set consisting of
$S$ together with the $r$-balls around all elements of $Z$ of distance at
most $r$ from $S$. 
\end{defn}

We have the following triviality

\begin{lemma}\label{enlargement_lemma}
Suppose $Y, Z \subset [G]^{< \infty}$ are such that $\diam(R) \leq r$ for
every $R \in Z$, for all distinct $R,R' \in Z$, $d(R,R') > 2 r$, and
$d(S,S') > 6 r$ for every $S,S' \in Y$. Then 
\begin{enumerate}
\item For every $Q \in B_r(Y,Z)$ there is an $S \in Y$ such that $S \subset
Q \subset B_{3r}(S)$. 
\item Every element of $B_{r}(Y,Z)$ is finite and connected, and the
elements of $B_r(Y,Z)$ are pairwise disjoint. 
\item If $R \in Z$ and $Q \in B_{r}(Y,Z)$, then either $B_{r}(R) \subset
Q$, or $d(R,Q) > r$. 
\end{enumerate} \qed
\end{lemma}

We now use Lemma~\ref{accumulation_lemma} to prove Theorem~\ref{cake}. 

\begin{proof}[Proof of Theorem~\ref{cake}]
Let $r_i = n 12^{i+1}$, and let $C_0, C_1, \ldots \subset X$ be Borel
maximal $4 r_i$-discrete sets satisfying the conclusion of
Lemma~\ref{accumulation_lemma}.
We now define sets $D_i \subset [G]^{< \infty}$. Given that we have defined
$D_j$ for $j < i$, we define $D_i$ by constructing a sequence $A^i_{0},
\ldots, A^i_i
\subset [G]^{< \infty}$, and letting $D_i = A^i_{i}$ at the end. To begin, let
$A^i_{0} = B_{r_{i}/4}(C_i)$. Hence, by the
definition of $C_i$, the elements of $A^i_{0}$ are connected, have
diameter $\leq r_{i}/2$, and for all distinct $R, R' \in A^i_{0}$ we have
$d(R,R') > 3r_i$.

For $0 < j \leq i$, let
\[A^i_{j} = B_{r_{i-j}}(A^i_{j-1},D_{i-j}).\]
We claim that at the end of this construction, the elements of $D_i = A^i_i$
have diameter $\leq r_i$ and are pairwise of
distance at least $2 r_i$. This is easy to prove by 
induction using Lemma~\ref{enlargement_lemma}, since for 
every $Q \in D_i = A^i_{i}$, there is some $S \in A^i_{0}$ such that 
\[Q \subset B_{3 r_{0}}(\ldots B_{3 r_{i-2}}(B_{3 r_{i-1}}(S)) \ldots )\]
and so $\diam(Q) \leq r_i/2 + 6 r_{i-1} + \ldots + 6 r_{0}$ which is at
most $r_i$ by our definition of $r_i$. 
Next, note that for each
$0 < j \leq i$, if $R \in D_{i-j}$ and $Q \in D_i$, then by induction using
Lemma~\ref{enlargement_lemma}, either $B_{r_{i-j}}(R) \subset Q$, or $d(Q,R) >
r_{i-j} - 3 r_{i - j - 1} - \ldots - 3 r_0$. Note that 
$r_{i-j} - 3 r_{i - j - 1} - \ldots - 3 r_0 \geq r_0$ for
all $0 < j \leq i$. 

To finish,
let $C = \{\tilde{S} : S \in \bigunion_i D_{i}\}$, where $\tilde{S}$ is the
set of vertices in $S$ together with all $x \in
X$ that are in some finite connected component of $G - \boundary S$.
It is clear that $\boundary \tilde{S} \subset \boundary S$, and
that $\boundaryinfty \tilde{S} = \boundary \tilde{S}$ as described in the
paragraph after the statement of Theorem~\ref{cake}. It follows from the
previous paragraph that all distinct $R, S
\in \bigunion_i D_i$ have $\boundary^n R \inters \boundary^n S = \emptyset$,
since $r_0 > 2n$. 
Finally, $\bigunion C = X$ since $\bigunion (\bigunion_i A^i_0) = X$ by our
choice of the sequence $C_i$.

\end{proof}

\end{appendices}

\end{document}